\documentclass[a4paper, 12pt]{amsart}
\usepackage{amssymb,amscd}
\usepackage[all]{xy}
\usepackage{graphicx}
\usepackage{cite}
\usepackage[table,xcdraw]{xcolor}
\usepackage{multirow}
\usepackage{hyperref}
\advance\hoffset by -1in
\advance\voffset by -1in
\newtheorem{theorem}{Theorem}

\newtheorem{lemma}{Lemma}[section]
\newtheorem{corollary}{Corollary}[theorem]
\newtheorem*{remark}{Remark}
\newtheorem*{definition}{Definition}

\makeatletter
\newcounter{bwt}

\makeatother
\DeclareMathOperator{\Tch}{\mathit{h}}

\DeclareMathOperator{\WeaklyLC}{\mathfrak{M}}
\DeclareMathOperator{\Weakly}{\mathsf{WN}}

\DeclareMathOperator{\Nv}{\mathsf{N}}

\newcommand{\WLC}{\WeaklyLC^{(2)}}
\newcommand{\We}{\Weakly^{(2)}}

\newcommand{\Nve}{\Nv^{(2)}}

\newcommand{\Ag}{{\mathfrak{A}}}

\newcommand{\Aeb}{{\mathcal B}}

\newcommand{\Bg}{{\mathfrak{B}}}

\title[Free weakly Novikov metabelian algebra of infinite rank]%
{Free weakly Novikov metabelian algebra\\ of infinite rank}

\author{Iritan Ferreira dos Santos, Alexey M. Kuz'min}

\begin{document}

\maketitle

\begin{abstract}
An explicit base with multiplication table is obtained for the free weakly Novikov metabelian algebra of infinite rank over an arbitrary field of characteristic~$\neq 2$.


\end{abstract}


\section{Introduction}
Weakly Novikov rings appeared in papers
by E.~Kleinfeld and H.~F.~Smith~\cite{kleinfeld1994generalization,kleinfeld1994right}
as the rings satisfying
\begin{align}
(x,y,z)&=(x,z,y)\quad\textit{(the right symmetry identity)},
\label{e-symm}\\
x(y,z,t)&=(y,z,xt)\quad\textit{(the weakly Novikov identity)},
\label{WN}
\end{align}
where $(x,y,z)=(xy)z-x(yz)$ stands, as usual, for the \textit{associator} in variables~$x,y,z$.
Note that~\eqref{WN} generalizes 
\begin{equation}
x(yz)= y(xz)\quad\textit{(the left commutativity identity)}.
\label{e-comm}
\end{equation}
For more results on nonmultilinear generalizations of identity~\eqref{WN},
see recent papers by Samanta and Hentzel~\cite{samanta2018third,samanta2019third}.

Let $F$ be a field of characteristic distinct from~$2$, 
$\Weakly$ be the variety of algebras over $F$ defined by~\eqref{e-symm} and~\eqref{WN}, and 
$\Nv$ be the Novikov subvariety of $\Weakly$, i.~e. the subvariety distinguished by~\eqref{e-comm}.
Following~\cite{kuz2017basic}, we set $\mathcal V^{(2)}$ to be the metabelian subvartiety of a given variety $\mathcal V$, i.~e. the subvariety distinguished by
\begin{equation}\label{met}
(xy)(zt)=0\quad\textit{(the metabelian identity)}.
\end{equation}

It's known from the paper by Shestakov and Zhang~\cite{shestakov2020solvability} that $\Nve$ is left nilpotent of index not more then~$9$. 
In the present paper, we obtain an explicit base with multiplication table for the free algebra of~$\We$ on the countable set of generators.
As a corollary, we prove that $\We$ and $\Nve$ are both left nilpotent of exact index~$5$. 

    We stress that questions on finding certain effective bases for free algebras in given varieties are strongly motivated by a number of open problems dealing with varieties of all alternative, Jordan, and Maltsev algebras~(see, for ex.~\cite{shestakov2006malcev,shestakov2006universal,shestakovAndZhukavets2007freealter,shestakovAndZhukavets2007freeMalc}).

\section{Preliminary results}
We set, as usual,
$[a,b]=ab-ba$,
$a\circ b=ab+ba$,
and
$ab=aR_b=bL_a$.
In what follows, 
$X=\{x_1,x_2,\dots\}$
stands for a countable set of generators of given free algebra. 
Unless otherwise stated, all algebras are considered over the field $F$ of characteristic~$\neq 2$.
Let $\mathfrak{M}$
be the variety of algebras
defined by~\eqref{WN}.

\subsection{Monomial base for the free algebra of $\WLC$}

\begin{theorem}\label{Th-Base-WLC}
The free algebra 
$\Ag=F_{\WLC}\left<X\right>$ 
possesses a basis formed by all monomials of the form
\begin{equation}\label{monom base}
x_iL_{x_{j_1}}\dots L_{x_{j_n}}R_{x_{k_1}}\dots R_{x_{k_t}},
\end{equation}
such that the variables
$x_{j_1},\dots,x_{j_n}$,
for
$n\geqslant 4$,
are invariant with respect to even permutations and the multiplication in 
$\left(\Ag,\cdot\right)$ is defined by the table
\begin{align*}
x_j\cdot x_i
&=x_iL_{x_j},\\
\left(x_iL_{x_j}\right)\cdot x_k
&= x_iL_{x_j}R_{x_k},\\
x_q\cdot \left(x_iL_{x_j}R_{x_k}\right)&=x_kL_{x_i}L_{x_j}L_{x_q}-x_kL_{x_q}L_{x_i}L_{x_j},\\
\left(x_iL_{x_{j_1}}\dots L_{x_{j_n}}\right)\cdot x_q
&=x_iL_{x_{j_1}}\dots L_{x_{j_n}}R_{x_{q}},\\
x_q\cdot \left(x_iL_{x_{j_1}}\dots L_{x_{j_n}}\right)
&=x_iL_{x_{j_1}}\dots L_{x_{j_n}}L_{x_q},\\
\left(x_iL_{x_{j_1}}\dots L_{x_{j_n}}R_{x_{k_1}}\dots R_{x_{k_t}}\right)\cdot x_q
&=x_iL_{x_{j_1}}\dots L_{x_{j_n}}R_{x_{k_1}}\dots R_{x_{k_t}}R_{x_{q}},
\end{align*}
where all possible nonzero products of base monomials are given.
\end{theorem}
\begin{proof}
Throughout this proof, we use metabelian identity~\eqref{met} with no comments.
Consider all possible restrictions of left commutativity identity~\eqref{WN} such that one of its variables is assumed to be taken form $\Ag^2$.
First, by
$x\in \Ag^2$, we have
\begin{equation}\label{rel-RLL}
\Ag^2 R_tL_zL_y=0.
\end{equation}
Further, by setting
$y\in \Ag^2$, we get
\[
\Ag^2 R_zR_tL_x=0.
\]
Moreover, in view of~\eqref{rel-RLL},
$z\in \Ag^2$, yields
\[
\Ag^2 [L_y,R_t]L_x=\Ag^2 L_yR_tL_x=0.
\]
Consequently,
\begin{equation}\label{rel-A3RL}
\Ag^3 R_tL_x=0.
\end{equation}
Furthermore, 
$t\in \Ag^2$ implies
\begin{equation}\label{rel-LLL-3-cycl}
\Ag^2 \left[L_zL_y,L_x\right]=0.
\end{equation}
Finally, by~\eqref{rel-RLL} and~\eqref{WN}, we obtain 
\[
xL_yL_zL_tL_u=-(z,y,x)L_tL_u=-(z,y,tx)L_u=xL_tL_yL_zL_u,
\]
i.~e.
\begin{equation}\label{rel-A-LLLL}
\Ag \left[L_yL_z,L_t\right]L_u=0.
\end{equation}
Therefore, by virute of~\eqref{rel-RLL}--\eqref{rel-A-LLLL}, $\Ag$ is spanned by the monomials of type~\eqref{monom base}.

To complete the proof it remains to check that identity~\eqref{WN} holds in $\left(\Ag,\cdot\right)$.
Indeed, on generators of $\Ag$, we have
\[
x_q\cdot \left(x_j,x_i,x_k\right)=x_q\cdot \left(x_iL_{x_j}R_{x_k}-x_kL_{x_i}L_{x_j}\right)=
-x_kL_{x_q}L_{x_i}L_{x_j}=\left(x_j,x_i,x_q\cdot x_k\right);
\]
otherwise,~\eqref{WN} is an immediate consequence of relations~\eqref{rel-RLL}-\eqref{rel-LLL-3-cycl}.
\end{proof}

\begin{corollary}\label{Th-Lie-Jord-Nilp}
Every Lie-nilpotent or Jordan-nilpotent subvariety of index~$n$ in $\WLC$ is nilpotent of index not more then~$2n+1$.
\end{corollary}

\begin{proof}
Let us prove that, for $\Bg\in \WLC$, 
$H_x=R_x-L_x$,
and
$\varTheta_x=R_x+L_x$, each of the relations 
\[
\Bg^2 H_{x_{1}}\dots H_{x_{n-1}}=0,
\quad
\Bg^2 \varTheta_{x_{1}}\dots \varTheta_{x_{n}-1}=0
\]
yields
$\Bg^{2n+1}=0$.
In the case $n=1$, the affirmation is trivial in view of~\eqref{rel-RLL}.
Further, we set
\[
R^n=R_{x_1}\dots R_{x_n},
\quad
L^n=R_{x_1}\dots L_{x_n},
\quad
H^n=H_{x_1}\dots H_{x_n},\quad n\geqslant 2.
\]
By virtue of~\eqref{rel-RLL} and~\eqref{rel-A3RL}, the relations 
\[
\Bg^2 H^nL_y=0,\quad
\Bg^3 R_yH^n=0
\]
yield, respectively,
\[
\Bg^2 L^{n+1}=0,\quad
\Bg^3 R^{n+1}=0.
\]
Therefore, by Theorem~\ref{Th-Base-WLC},
\[
\Bg^3 T_{x_1}\dots T_{x_{2n}}=0,
\]
for all
$T_{x_i}\in\left\{R_{x_i},L_{x_i}\right\}$.
\end{proof}

\begin{corollary}\label{Th-Alt-Symm-Nilp}
The flexible and the antiflexible subvarieties of $\WLC$ are both nilpotent of index not more than~$5$.
Moreover, for $\Bg\in \WLC$ and $n\geqslant 2$, the identity
\begin{equation}\label{eq-weak-flex}
\left(\Bg^n,x,y\right)
=\pm\left(y,x,\Bg^n\right)
\end{equation}
implies
$\Bg^{n+3}=0$.
\end{corollary}
\begin{proof}
By virtue of~\eqref{met}, identity~\eqref{eq-weak-flex} yields
\[
\Bg^nR_xR_y=\mp \Bg^nL_xL_y.
\]
Thus, in view of~\eqref{rel-RLL}, we have
\[
\Bg^nL_xL_yR_z=\mp\Bg^nR_xR_yR_z=\Bg^nR_xL_yL_z=0.
\]
Similarly, by~\eqref{rel-A3RL}, we obtain
\[
\Bg^nL_xR_yR_z=\mp\Bg^nL_xL_yL_z=\Bg^nR_xR_yL_z=0.
\]
Therefore, by Theorem~\ref{Th-Base-WLC},
$\Bg^{n+3}=0$.
\end{proof}

\section{Main Theorem}

\subsection{Auxiliary identities}

Let
$\Ag=F_{\We}\left<X\right>$ 
be the free algebra of the variety $\We$ on the countable set 
$X=\{x_1,x_2,\dots\}$ of generators;
the symbol $T_x\in\{R_x,L_x\}$ is used as a common denotation for the 
operators
$R_x$ and $L_x$
of right and left multiplication by~$x\in X$, respectively.

\begin{lemma}\label{Lm_WN_def}
The free algebra $\Ag$ satisfies the identities
\begin{align}
\label{We'}
x\cdot(y,z,t)&=\left(y,xz,t\right),\\
\label{We''}
(x,y,zt)&=(x,zy,t),\\
\label{We'''}
x\cdot y(zt)&=-(x,zy,t),\\
\label{We'v}
x(yz)\cdot t&=\bigl(x,[y,z],t\bigr)+(y,xz,t),\\
\label{We v}
\left(x,y,z\right)\cdot t&=\Tch(x,y,z,t)+\bigl(x,y\circ z,t\bigr),
\end{align}
where
\[
\Tch(x,y,z,t)\stackrel{\mathrm{def}}{=}
(xy,z,t)-(y,xz,t)-2(x,yz,t).
\]
\end{lemma}
\begin{proof}
By virtue of~\eqref{WN} and~\eqref{e-symm}, we have
\[
x\cdot (y,z,t)= x\cdot (y,t,z)=(y,t,xz)=(y,xz,t).
\]
Thus,~\eqref{We'} is proved. 
Further, combining~\eqref{WN} and~\eqref{We'}, we get~\eqref{We''}:
\[
(x,y,zt)=z\cdot (x,y,t)=(x,zy,t).
\]
Furthermore, by~\eqref{met} and~\eqref{We''}, we obtain~\eqref{We'''}: 
\[
x\cdot y(zt)=-(x,y,zt)=-(x,zy,t).
\]
Now, recall that every nonassociative ring satisfies the Teichm\"uller's identity 
\[
x\cdot (y,z,t)+\left(x,y,z\right)\cdot t
=(xy,z,t)-(x,yz,t)+(x,y,zt).
\]
Hence,  using~\eqref{We'} and~\eqref{We''}, we have
\[
\left(x,y,z\right)\cdot t
=(xy,z,t)-(y,xz,t)
-\bigl(x,\left[y,z\right],t\bigr).
\]
Consequently, to prove~\eqref{We'v} it remains to observe that
\[
\left(x,y,z\right)\cdot t-\left(xy,z,t\right)
=-x(yz)\cdot t,
\]
in view of~\eqref{met}.
Moreover, taking into account the denotations, introduced above,
\begin{multline*}
\left(x,y,z\right)\cdot t
=(xy,z,t)-(y,xz,t)
+\bigl(x,\left[z,y\right],t\bigr)=\\
=(xy,z,t)-(y,xz,t)-2(x,yz,t)
+\bigl(x,y\circ z,t\bigr)=\\
=\Tch(x,y,z,t)+\bigl(x,y\circ z,t\bigr),
\end{multline*}
we complete the proof of~\eqref{We v}.
\end{proof}

\begin{lemma}\label{lm monom A5}
A monomial
$\mu\in \Ag^5$
can be nonzero only if it has the form
$\mu=wR_xR_yR_z$
for some
$w\in \Ag^2$ 
and $x,y,z\in X$
\end{lemma}
\begin{proof}
Let us consider consecutively all possible cases for 
$\mu\neq wR_xR_yR_z$.
First, by~\eqref{WN} and~\eqref{met}, we have
\begin{align}
\label{RRL=0}
 wR_xR_yL_z&=z(w,x,y)-(w,x,zy)=0,\\
\label{RLL=0}
wR_xL_yL_z&=w(z,y,x)-(z,y,wx)=0.
\end{align}
Similarly, by~\eqref{met} and~\eqref{We'}, we get
\begin{equation}
\label{LLL=0}
wL_xL_yL_z=(y,zx,w)-z(y,x,w)=0.
\end{equation}
Further, combining~\eqref{WN},~\eqref{met},~\eqref{We'} with~\eqref{RRL=0} and~\eqref{RLL=0}, we proceed 
\begin{align}
\label{RLR=0}
wR_xL_yR_z&=wR_xR_zL_y+(y,wx,z)-w(y,x,z)=0,\\
\label{LRL=0}
wL_xR_yL_z&=wR_yL_xL_z+z(x,w,y)-(x,w,zy)=0.
\end{align}
Now, by~\eqref{We'},~\eqref{RLL=0}, and~\eqref{LRL=0}, we obtain
\begin{equation}
\label{LLR=0}
wL_xL_yR_z=wL_xR_zL_y+w[L_y,R_z]L_x+(y,xw,z)-x(y,w,z)=0.
\end{equation}
Finally, using~\eqref{RLR=0},~\eqref{LLR=0}, and taking into account~\eqref{e-symm}, we prove
\begin{equation}
\label{LRR=0}
wL_xR_yR_z=wR_yL_xR_z-
 wL_yL_xR_z+(x,w,y)z-
(x,y,w)z=0.
\end{equation}
\end{proof}

\begin{lemma}\label{lm Asocc 4 weak id}
The T-ideal 
$\bigl<(a,bc,d)\bigr>^{\mathrm{T}}$
is spanned by the associators 
$(x_i,x_jx_k,x_{\ell})$
taken only on generators
$x_i,x_j,x_k,x_\ell\in X$
of
$\Ag$.
\end{lemma}
\begin{proof}
It's enough to notice that,
by Lemma~\ref{lm monom A5},
the associators
$(x_i,x_jx_k,x_{\ell})$
lie in the annihilator of~$\Ag$
and all associators of the form 
\[
(w,x_jx_k,x_{\ell}),\quad (x_i,wx_k,x_{\ell}),\quad (x_i,x_jw,x_{\ell}), \quad(x_i,x_jx_k,w),
\]
for 
$w\in {\Ag}^2$,
are null ones.
\end{proof}

\subsection{Pre-base}

Recall that a function
$f(x,y)$ 
is said to be symmetric w.r.t.~$x,y$ if
\[
f(x,y)=f(y,x).
\]

\begin{definition}
{\normalshape
The \textit{base elements of $\Ag^3$} are all elements  \textit{of types}~\eqref{bwt-i}--\eqref{bwt-v} defined by the list below, for $x,y,z,t_1,\dots,t_k\in X$:
\begin{list}{(\roman{bwt})}{\usecounter{bwt}}
\item\label{bwt-i} 
$x(yz)$,
\item\label{bwt-ii} 
$(x,t_1,t_2)$,
\item\label{bwt-iii} 
$\left(x,yt_1,t_2\right)$,
\item\label{bwt-iv} 
$\Tch(x,t_1,t_2,t_3)$,
\item\label{bwt-v} 
$\left(xt_1\right)R_{t_2}\ldots R_{t_k}$,
\end{list}
where the indices~$t_i$ are used each time when the corresponding element possesses, by definition, the symmetry property w.r.t. all the subset $\left\{t_i\right\}$ of its variables.}
\end{definition}
\begin{lemma}\label{lm pre-base}
The ideal~$\Ag^3$ is spanned by its base elements. 
\end{lemma}

\begin{proof}
First,
taking into account the defining identity~\eqref{e-symm} for $\We$,
it's not hard to see that the subspace of polynomials of degree~$3$ in~$\Ag$ is spanned by the elements of types~\eqref{bwt-i} and~\eqref{bwt-ii}, in view of the trivial equality
\[
(xy)z=(x,y,z)+x(yz).
\]
Further, identities~\eqref{We'},~\eqref{We'''}--\eqref{We v}, yield that
the subspace of polynomials of degree~$4$ in~$\Ag$ is spanned by the elements of types~\eqref{bwt-iii} and~\eqref{bwt-iv}.
While that the symmetry property established for the elements of type~\eqref{bwt-iii} is a direct consequence of identity~\eqref{We'}. 
At the same time,
in the case of elements of type~\eqref{bwt-iv}, the required symmetry property can be obtained as follows.
On one hand, the polynomial 
$\Tch(x,y,z,t)$
is defined as an symmetric element of~$z,t$.
On the other hand, identity~\eqref{We v} implies 
the symmetry of $\Tch(x,y,z,t)$ with respect to~$y,z$.

Furthermore, by Lemma~\ref{lm monom A5},
the subspace of polynomials of degree~$\geqslant 5$ in~$\Ag$ is spanned by
$R$-words.
Moreover, combining the defining identities~\eqref{e-symm} and~\eqref{met} of $\We$, we have
\[
wR_yR_z=wR_zR_y,\quad w\in \Ag^2.
\]
Finally, Lemma~\ref{lm Asocc 4 weak id} implies
\[
\Tch(x,t_1,t_2,t_3)\cdot y=\left(x t_1\right)R_{t_2}R_{t_3}R_y.
\]
This proves the required symmetry property of the elements of type~\eqref{bwt-v}.
\end{proof}

\subsection{Base}

\begin{theorem}\label{Thrm Base}
The free algebra 
$\Ag$ 
possesses a base formed by all elements from~$X$,~$X^2$, and all base elements of $\Ag^3$
equipped with the multiplication~$\cdot$ introduced by the following rules: 
\begin{itemize}
\item 
Every nonzero left action of a generator
$x\in X$ 
on a base element is defined by one of the following formulas:
\begin{align*}
x\cdot y&=xy,\\
x\cdot yz&=x(yz),\\ 
x\cdot y(zt)&=-(x,zy,t),\\ 
x\cdot(y,t_1,t_2)&=\left(y,xt_1,t_2\right),
\end{align*}
where $y,z,t,t_1,t_2\in X$.
\item
Every nonzero right action of a generator
$y\in X$ 
on a base element of degree~$\geqslant2$ is defined by one of the following formulas:
\begin{align*}
xz\cdot y&=(x,z,y)+x(zy),\\
x(zt)\cdot y&=\bigl(x,[z,t],y\bigr)+(z,xt,y),\\
\left(x,t_1,t_2\right)\cdot y&=\Tch(x,t_1,t_2,y)+\bigl(x,t_1\circ t_2,y\bigr),\\
\Tch(x,t_1,t_2,t_3)\cdot y&=\left(x t_1\right)R_{t_2}R_{t_3}R_y,\\
\left(x t_1\right)R_{t_2}\ldots R_{t_k}\cdot y&=
\left(x t_1\right)R_{t_2}\ldots R_{t_k}R_y,\quad k\geqslant 4,
\end{align*}
where $x,z,t,t_1,\dots,t_k\in X$.
\item
All the other products of base elements not listed in the formulas above are null once by definition.
\end{itemize}

\end{theorem}
\begin{proof}
First we stress that the multiplication~$\cdot$ is well-defined.
Indeed, by the direct verification, one can confirm that 
all symmetry properties,
required for the elements of types~\eqref{bwt-ii}--\eqref{bwt-v},
are inherited by the corresponding right parts in formulas of multiplication table.

Let $\Aeb$ be the span of all introduced base elements over the field~$F$. Consider $\Aeb$ as an algebra with respect to the multiplication~$\cdot$ defined by the rule of Theorem~\ref{Thrm Base}.
Then, Lemma~\ref{lm pre-base} yields that $\Ag$ could be isomorphic to some quotient of
$\Aeb$.
Let us prove that actually $\Ag$ is isomorphic to $\Aeb$ under the isomorphism induced by the identical mapping
$X\mapsto X$.

Our proof is by formal checking of defining identities~\eqref{WN}--\eqref{met} of the variety~$\We$ for the algebra $\Aeb$. 
First, metabelian identity~\eqref{met} holds in~$\Aeb$ immediately, by definition of the multiplication~$\cdot$.
Further, let us set
\[
\bigl(a,b,c\bigr)_{\Aeb}\stackrel{\mathrm{def}}{=}(a\cdot b)\cdot c-a\cdot (b\cdot c).
\]
In what follows, we assume that
$x,y,z,t,u$ are variables from $X$ and $w$ is a common denotation for base elements of $\Aeb$ such that
neither 
$w$ nor $X\cdot w$, $w\cdot X$
lie in the annihilator
\[
\mathrm{Annh}(\Aeb)=F\cdot (X,X^2,X) 
\]
of $\Aeb$, i.~e.
\[
w\in \bigl\{x,\,xy,\, (x,y,z),\,\Tch(x,y,z,t),\, (xy)R_zR_t\dots R_u\bigr\}.
\]
Let us check the right symmetry of $\Aeb$: 
\[
\bigl(a,b,c\bigr)_{\Aeb}=\bigl(a,c,b\bigr)_{\Aeb}.
\]
First, we compute all nonzero associators of type
$\bigl(w,b,c\bigr)_{\Aeb}$, for $b,c\in X$,
and verify that the obtained elements are 
symmetric in~$b,c$.
Indeed,
\begin{align*}
\bigl(x,b,c\bigr)_{\Aeb}
&=xb\cdot c-x\cdot bc=(x,b,c)+x(bc)-x(bc)=(x,b,c),\\
\bigl(xy,b,c\bigr)_{\Aeb}
&=\bigl(xy\cdot b\bigr)\cdot c=\bigl((x,y,b)+x(yb)\bigr)\cdot c=\\
&=\Tch(x,y,b,c)+\bigl(x,y\circ b,c\bigr)
+\bigl(x,[y,b],c\bigr)+(y,xb,c)=\\
&=\Tch(x,y,b,c)+\bigl(x,yb,c\bigr)
+\bigl(x,yb,c\bigr)+(y,xb,c),\\
\bigl((x,y,z),b,c\bigr)_{\Aeb}
&=\bigl((x,y,z)\cdot b\bigr)\cdot c=\\
&=\Bigl(\Tch(x,y,z,b)+\bigl(x,y\circ z,b\bigr)\Bigr)\cdot c=(xy)R_zR_bR_c,\\
\bigl(\Tch(x,y,z,t),b,c\bigr)_{\Aeb}
&=\bigl(\Tch(x,y,z,t)\cdot b\bigr)\cdot c=(xy)R_zR_tR_bR_c,\\
\bigl((xy)R_zR_t\dots R_u,b,c\bigr)_{\Aeb}
&=\bigl((xy)R_zR_t\dots R_u\cdot b\bigr)\cdot c=(xy)R_zR_t\dots R_uR_bR_c.
\end{align*}
Further, we compare the values of associators in pairs of types
\[
\bigl(a,b,w\bigr)_{\Aeb}\,,\quad  \bigl(a,w,b\bigr)_{\Aeb}\,,\quad a,b\in X,\quad w\notin X.
\]
In the case $w\in X^2$, we have
\begin{align*}
\bigl(a,b,xy\bigr)_{\Aeb}
&=-a\cdot(b\cdot xy)
=-a\cdot b(xy)=(a,xb,y)
=(a,xy,b),\\
\bigl(a,xy,b\bigr)_{\Aeb}
&=(a\cdot xy)\cdot b-a\cdot (xy\cdot b)=a(xy)\cdot b-a\cdot\bigl((x,y,b)+x(yb)\bigr)=\\
&=\bigl(a,[x,y],b\bigr)+(x,ay,b)-(x,ay,b)+(a,yx,b)
=(a,xy,b).
\end{align*}
Furthermore, let ${\Aeb}^n$ be the ideal of
$\Aeb$ spanned by all base elements of degree~$\geqslant n$. 
We observe that all terms of associators of types
\[
\bigl(X,X,{\Aeb}^3\bigr)_{\Aeb},\quad
\bigl(X,{\Aeb}^3,X\bigr)_{\Aeb}
\]
are null, by virtue of the relations
\[
X\cdot {\Aeb}^4=0,\quad
X\cdot (X^2X)\subseteq\mathrm{Annh}(\Aeb),\quad
X\cdot (X,X,X)\subseteq\mathrm{Annh}(\Aeb).
\]
Therefore,~$\Aeb$ is right symmetric.

Finally, to complete the proof, we verify that $\Aeb$ is weakly Novikov. 
Indeed, in view of the relations above, all terms of the identity
\[
a\cdot\bigl(b,c,d\bigr)_{\Aeb}=\bigl(b,c,a\cdot d\bigr)_{\Aeb}\,
\]
are null whenever at least one of its variables $a,b,c,d$ takes a value that doesn't lie in~$X$.  
Otherwise, we have,
\begin{align*}
x\cdot\bigl(y,z,t\bigr)_{\Aeb}&=x\cdot(y,z,t)=(y,xz,t),\\
\bigl(y,z,xt\bigr)_{\Aeb}&=-y\cdot(z\cdot xt)=-y\cdot z(xt)=(y,xz,t).
\end{align*}
\end{proof}

\section{Corollaries}

We stress that combining Lemma~\ref{lm monom A5} and 
the established above symmetry property for $R$-words
\[
wR_yR_z=wR_zR_y,\quad w\in \Ag^2
\]
with the Medvedev's two-term identity theorem~\cite{Medvedev78},
one can prove that the variety
$\We$ is Spechtian~(see, also,~\cite{kuz2008spechtian,platonova2004varieties,pchelintsev2008identities,Umirbaev85}).

In this section, we describe defining identities for nilpotent and non-nilpotent subvarieties of~$\We$.

\subsection{Nilpotent subvarieties}

\begin{lemma}\label{Coroll_deg5}
Every proper subvariety of $\We$ distinguished by some multilinear identity of degree~$n\geqslant 5$ is nilpotent of index not more then~$n+1$.
\end{lemma}

\begin{proof}
Consider a multilinear polynomial 
$f=f(x_1,\dots,x_n)\in \Ag$
on~$n\geqslant 5$ variables.
By Theorem~\ref{Thrm Base},
$f$
can be written down as
\[
f=\sum_{i=1}^n
\lambda_i\left(x_ix_{\sigma_i(2)}\right)R_{x_{\sigma_i(3)}}\ldots R_{x_{\sigma_i(n)}},
\]
where all
$\lambda_i\in F$
and every
$\sigma_i$
is a permutation on the set
$\{1,2,\dots,n\}$
defined by the rule
\begin{equation}\label{property sigma-i}
\sigma_i(1)=i,\quad \sigma_i(2)<\sigma_i(3)<\dots<\sigma_i(n).
\end{equation}
In other words, rule~\eqref{property sigma-i} states that
$\sigma_1=\mathrm{id}$
and
\[
\sigma_i={(1\,2\dots i)}^{-1}, \text{ for } i=2,3,\dots,n.
\]
Suppose that $\mathcal V_f$ is a proper subvariety of $\We$ distinguished by
the identity
$f$.
Then, with no loss of generality, one may assume 
$\lambda_1=1$ 
and set
$x_1=w\in X^2$.
Hence, by Lemma~\ref{lm monom A5},
\[
f(w,x_2\dots,x_n)=wR_{x_2}\dots R_{x_n}.
\]
Therefore, by Theorem~\ref{Thrm Base},
$\mathcal V_f$
is nilpotent of index not more than~$n+1$.
\end{proof}

\begin{remark}
{\upshape In particular, Lemma~\ref{Coroll_deg5} states that the variety 
$\We$ 
has the topological rank\footnote{See~\cite{Drensky-Rashkova89,kuz2015topol,Pchelintsev81,platonova2006varieties}.}
$r_t\bigl(\We\bigr)=2$.}
\end{remark}

\begin{lemma}\label{Coroll_deg2}
Every proper subvariety of $\We$ distinguished by some multilinear identity of degree~$2$ is nilpotent of index not more then~$5$.
\end{lemma}
\begin{proof}
Indeed, for
\[
f(x_1,x_2)=x_1x_2+\lambda\,x_2x_1,\quad \lambda\in F,
\]
in view of~\eqref{RLL=0}, we have
\[
f\bigl((x_1x_3)R_{x_4}R_{x_5},x_2\bigr)=(x_1x_3)R_{x_4}R_{x_5}R_{x_2}.
\]
Hence, Theorem~\ref{Thrm Base} yield that a subvariety of $\We$ distinguished by $f$ should be nilpotent of index not more then~$5$.
\end{proof}

\subsection{Non-nilpotent subvarieties}

\begin{lemma}\label{Coroll_deg4}
If $\mathcal V_f$ is a proper non-nilpotent subvariety of $\We$ distinguished by some multilinear identity 
$f=f(x_1,x_2,x_3,x_4)$ of degree~$4$, 
then $f$ should have the form
\[
f=\sum_{\delta\in\mathrm{A}_4}
\lambda_{\delta}\left(x_{\delta(1)},x_{\delta(2)}x_{\delta(3)},x_{\delta(4)}\right),
\]
where
all
$\lambda_{\delta}\in F$
and
$\mathrm{A}_4$
is the alternating group on the set 
$\{1,2,3,4\}$.
\end{lemma}
\begin{proof}
By Theorem~\ref{Thrm Base}, $f$ can be written down as
\[
f=\sum_{\delta\in\mathrm{A}_4}
\lambda_{\delta}\left(x_{\delta(1)},x_{\delta(2)}x_{\delta(3)},x_{\delta(4)}\right)
+
\sum_{i=1}^4
\mu_i\Tch(x_i,x_{\sigma_i(2)},x_{\sigma_i(3)},x_{\sigma_i(4)}),
\quad 
\lambda_{\delta},\mu_i\in F,
\]
where
$\sigma_i$ 
is the permutation defined by rule~\eqref{property sigma-i}, for $n=4$.
Suppose that at least one of the scalars~$\mu_i$ is nonzero.
Then, with no loss of generality, it's enough to fix $\mu_1=1$ and set
$x_1=w\in X^2$.
Hence, by Lemma~\ref{lm monom A5}, we get
\[
f(w,x_2,x_3,x_4)=wR_{x_2}R_{x_3}R_{x_4}.
\]
However, in view of Theorem~\ref{Thrm Base}, this contradicts with the assumption of non-nilpotency for $\mathcal V_f$.
The obtained contradiction completes the proof.
\end{proof}

\begin{lemma}\label{Coroll_deg3}
If $\mathcal V_f$ is a proper non-nilpotent subvariety of $\We$ distinguished by some multilinear identity 
$f=f(x_1,x_2,x_3)$ of degree~$3$, 
then $f$ should have the form
\[
f=\sum_{\sigma\in\mathrm{S}_3}
\lambda_{\sigma}\,x_{\sigma(1)}\left(x_{\sigma(2)}x_{\sigma(3)}\right),
\quad 
\lambda_{\sigma}\in F,
\]
where
$\mathrm{S}_3$
is the symmetric group on the set 
$\{1,2,3\}$.
\end{lemma}
\begin{proof}
Applying, as above, Theorem~\ref{Thrm Base}, we write down $f$ as
\[
f=\sum_{\sigma\in\mathrm{S}_3}
\lambda_{\sigma}\,x_{\sigma(1)}\left(x_{\sigma(2)}x_{\sigma(3)}\right)
+
\sum_{\delta\in\mathrm{C}_3}
\mu_{\delta}\left(x_{\delta(1)},x_{\delta(2)},x_{\delta(3)}\right),
\quad 
\lambda_{\sigma},\mu_{\delta}\in F,
\]
where
$\mathrm{C}_3$
is the cyclic group on the set 
$\{1,2,3\}$.
If there is at least one of the scalars~$\mu_{\delta}\neq 0$,
then, similarly to the above cases, we restrict with the assumption
$\mu_{\mathrm{id}}=1$.
Then, applying Lemma~\ref{lm monom A5}, we obtain
\[
f\bigl((x_1x_4)x_5,x_2,x_3\bigr)=(x_1x_4)R_{x_5}R_{x_2}R_{x_3}.
\]
Again by virtue of Theorem~\ref{Thrm Base}, we have the contradiction with the hypothesis of non-nilpotency of~$\mathcal V_f$. 
\end{proof}
\subsection*{Acknowledgments}%
The results of this paper were first presented at the scientific seminar of the research group "Algebra and the Universal Algebraic Geometry" of the Federal University of the Rio Grande do Norte, Brazil.  
The authors are very thankful to all participant of the seminar, especially to Nir Cohen, Arkady Tsukov, Elena V. Aladova, Alexander S. Sivatski, Alan de Araújo Guimarães, Ana Beatriz Gomez da Silva, and José Victor Gomes Teixeira for the creative work atmosphere and the constantly useful discussions.

\bibliography{reference}

\bibliographystyle{abbrvurl}


 

\end{document}